\documentclass[preprint,12pt]{elsarticle}

\usepackage[normalem]{ulem} 

 \usepackage{xcolor}
\usepackage{hyperref}

\usepackage[left=2cm,right=2cm,top=2cm,bottom=2cm]{geometry}

\usepackage{amssymb}

\usepackage{amsmath,amsthm}
\DeclareMathAlphabet{\mathbfsf}{\encodingdefault}{\sfdefault}{bx}{n}
\numberwithin{equation}{section}

\usepackage{enumitem}
\newlist{steps}{enumerate}{1}
\setlist[steps, 1]{label = Step \arabic*:}

\newtheorem{thm}{Theorem}
\newtheorem{lem}[thm]{Lemma}

\newdefinition{definition}{Definition}
\newdefinition{rmk}{Remark}
\newdefinition{exmp}{Example}
\newproof{pf}{Proof}
\newproof{rmk-litt}{{\protect\textit{Remark on the litterature}}}

\newcommand {\corpsValeurs} {\mathbb{C}}
\newcommand {\corpsC} {\mathbb{C}}
\newcommand {\corpsR} {\mathbb{R}}

\newcommand{\vertiii}[1]{{\left\vert\kern-0.25ex\left\vert\kern-0.25ex\left\vert #1 
        \right\vert\kern-0.25ex\right\vert\kern-0.25ex\right\vert}}

\journal{---} 

\begin{document}

\begin{frontmatter}


\title{Integral representation formula for linear non-autonomous difference-delay equations} 

\author{L. Baratchart}
\author{S. Fueyo \footnote{Corresponding author.}}
\author{J.-B. Pomet}
\address{Université Côte d'Azur, Inria, Teams FACTAS and MCTAO,
  \\ 2004, route des Lucioles, 06902 Sophia Antipolis, France.}

\begin{abstract}
 This note states and proves an integral representation formula of the ``variation-of-constant'' type for  continuous solutions
  of linear non-autonomous difference delay systems, in terms of a Lebesgue-Stieltjes
  integral involving a fundamental solution and the initial data of the system.
  This gives a precise and correct version of several
  formulations appearing in the literature, and extends them to the
  time-varying case.
  This is of importance for further stability studies of various kinds of delay systems.
\end{abstract}


\end{frontmatter}

\section{Introduction and statement of the result}
\label{sec_resul}

Transport phenomena commonly occur in practical situations in biology or physics, for example to model  the migration of cells in an organism,
  as well as road traffic networks or high-frequency signals traveling through an electrical circuit.
  These phenomena are typically modelled by delay equations.
Basic properties thereof, such as stability, stabilization
and controllability, have been widely considered in the literature
\cite{Suarez,rihan2021delay,gu2003stability,loiseau2009topics}.
Reference \cite{Suarez} deals with applications to electrical engineering,
in particular  the
stability of microwave circuits, which was the
initial motivation of the authors to study time-periodic systems described by equations
such as \eqref{system_lin_formel} below; in fact, the latter
constitutes a natural
model for  the high-frequency limit system of a microwave circuit, whose stability is crucial to the one  of the circuit itself 
(see \cite{SFueyo,Soumis-SIMA}).

In this note,  we consider a linear non-autonomous difference-delay system of the form:
\begin{eqnarray}
\label{system_lin_formel}
  y(t)=\sum_{j=1}^ND_j(t)y(t-\tau_j), \qquad t > s,\qquad
y(s+\theta)=\phi(\theta)\  \mathrm{for}\  -\tau_N\leq \theta \leq0,
\end{eqnarray}
where 
$d$, $N$ are positive integers and
the delays $\tau_1 < \cdots < \tau_N$  are strictly positive real numbers,
while $D_1(t),\ldots,D_N(t)$ are complex\footnote{We treat here the case of
  complex coefficients (in
  the matrices $D_j$ as well as in the solutions); real coefficients
  can be handled
  in exactly the same way.}
$d\times d$ matrices depending continuously on time $t$,
the real number $s\in \mathbb{R}$ is the initial time
and the function $\phi:[-\tau_N,0]\to\corpsValeurs^d$  the
\emph{initial data}.
By writing $t>s$, we imply that this system is understood in forward time only. To deal with backward systems, one would have to make assumptions on the invertibility of the map $t \mapsto D_{\tau_N}(t)$.

A \emph{solution} is a map $y:[s-\tau_N,+\infty)\to\corpsValeurs^d$
that satisfies \eqref{system_lin_formel}.
  Clearly, in order to get a continuous solution, the initial data 
  which is but the restriction of $y$  to the interval $[-\tau_N,0]$
  must satisfy a compatibility condition, namely
$\phi$ should belong to the set $C_s$ defined by
\begin{eqnarray}
C_s :=\{ \phi \in  C^0([-\tau_N,0],\corpsValeurs^d):\,\phi(0)=\sum\limits_{j=1}^ND_j(s)\phi(-\tau_j)\},
\end{eqnarray}
where $C^0(E,F)$ indicates the set of continuous maps from $E$ into $F$.
Conversely, given $\phi$ in $C_s$, an easy recursion shows that System~\eqref{system_lin_formel} has 
a unique solution $y$ with initial data $\phi$, and that this solution is continuous.


The goal of this note is to give a precise statement, as well as a detailed proof, of the integral formula, often called representation formula (see \textit{e.g.} \cite{Hale}), expressing the  solution to System~\eqref{system_lin_formel} in terms of the initial condition $\phi\in C_s$ and the so-called \emph{fundamental solution}, that can be
  viewed as a particular (non-continuous) matrix-valued solution of \eqref{system_lin_formel}.
By definition, the fundamental solution is 
the map $X:\mathbb{R}^2\to\mathbb{C}^{d\times d}$ satisfying
\begin{eqnarray}  
 \label{solution_fondamentale}
X(t,s)=\left\{
 \begin{array}{ll}
 0 \mbox{ for $t<s$}, \\
I_d+\sum\limits_{j=1}^N D_j(t) X(t-\tau_j,s) \mbox{ for $t \ge s$},
\end{array}
\right.
 \end{eqnarray}
where $I_d$ denotes the $d\times d$ identity matrix.
Arguing inductively, it is easy to check that $X$ uniquely exists and is
continuous at each $(t,s)$ such that $t-s \notin \mathcal{F}$, where
$\mathcal{F}$ is the positive lattice in $[0,+\infty)$ generated by the numbers $\tau_\ell$:
\begin{eqnarray}
\label{set_discontinuities}
\mathcal{F}:= \left\{ \sum\limits_{\ell=1}^N n_\ell \, \tau_\ell \,,\; (n_1,\ldots,n_N)\in\mathbb{N}^N\right\}\,.
\end{eqnarray}
Clearly, $X$ has a  bounded jump across each line $t-s=\mathfrak{f}$ for
$\mathfrak{f}\in\mathcal{F}$; in fact, a moment's thinking will convince the reader that $X$ has the form
    \begin{equation}
    \label{eq:5}
    X(t,s)
    \,=\, - \sum_{\mathfrak{f}\in\,[0,t-s]\,\cap\,\mathcal{F}}\,\mathfrak{C}_\mathfrak{f}(t)\,,
    \quad s\leq t\,,
  \end{equation}
  where each
  $\mathfrak{C}_\mathfrak{f}(.)$ is continuous\footnote{
      Additional smoothness assumptions on the maps $D_j(.)$ would
      transfer to the $\mathfrak{C}_\mathfrak{f}$.  
One can also see that  $\mathfrak{C}_\mathfrak{f}(t)$ is a finite sum of products of
$D_j(t-\mathfrak{f}')$,
where $\mathfrak{f}'$ ranges over the elements of $\mathcal{F}$ whose
defining integers $n_\ell$ in \eqref{set_discontinuities} do not exceed
those defining $\mathfrak{f}$, the
empty product being the identity matrix.
A precise expression for $\mathfrak{C}_\mathfrak{f}$ can be obtained by
reasoning as in \cite[Sec. 3.2]{Chitour2016} or \cite[Sec. 4.5]{BFLP}, but we will not need it.
}
$\corpsR\to\corpsValeurs^{d\times d}$.
The announced representation formula is now given by the following
theorem, whose proof can be found in Section \ref{proof_the_fond}.
\begin{thm}[representation formula]
  \label{prop:fondamental}
  For  $s\in\corpsR$ and  $\phi$ in $C_s$ the 
  solution $y\in C^0([s-\tau_N,+\infty),\corpsValeurs^{d})$ to \eqref{system_lin_formel} is given by
  \begin{eqnarray}
  \label{formula_representation}
y(t)=-\sum\limits_{j=1}^N \int_{s^{-}}^{(s+\tau_j)^{-}} d_{\alpha}X(t,\alpha)D_j(\alpha)\phi(\alpha-\tau_j-s),\qquad t \ge s\,,
  \end{eqnarray}
  where $X$ was defined in \eqref{solution_fondamentale}. 
\end{thm}
The integrals $\int_{s^-}^{(s+\tau_j)^-}$ in
Equation~\eqref{formula_representation} are understood as Lebesgue-Stieltjes integrals on the intervals $[s,s+\tau_j)$. They are well defined because,
for fixed $t$, the function $X(t,\cdot)$ is locally of bounded variation.
Basic facts regarding Lebesgue-Stieltjes integral and
functions of bounded variation are recalled in Section~\ref{sub:sec1}, for the ease of the reader.

\section{Further motivation  and comments}
Representation formulas like \eqref{formula_representation} are
fundamental to deal with linear functional dynamical systems.
In particular,  they offer
  a base to devise stability  criteria
{\it via} a frequency domain approach, using Laplace transforms:
in \cite{Henry1974}, \cite{Hale} or in the recent
  manuscript~\cite{Soumis-SIMA} by the  authors,
exponential stability is investigated this way for  difference-delay system and various
functional differential equations, either time-invariant or time-varying.
In fact, \eqref{formula_representation} entails that the solutions of System~\eqref{system_lin_formel} are exponentially stable if the total variation of the function $X(t,.)$ on $[s,s+\tau_N]$ is bounded by $ce^{-\alpha (t-s)}$ for some strictly positive  $c$ and $\alpha$. Hence, a careful study of  $X(t,.)$ 
(for example through an analysis of the $\mathfrak{C}_\mathfrak{f}$
in \eqref{eq:5}) yields important information on the
exponential stability properties of that system. These results, in turn,  are relevant to the stability of 1-D
hyperbolic PDE's; see \textit{e.g.} \cite[Theorem 3.5 and Theorem
3.8]{bastin2016stability} and \cite{BARATCHART2022228}.

Though formulas of the type \eqref{formula_representation}
appear at several places in the literature for various classes of linear
dynamical systems, the purpose of the present note
is to state it  carefully and
prove it in detail.
Our motivation here is twofold.
On the one hand, the time-varying case has apparently
not been treated.
On the other hand, and more importantly perhaps, we found that,
even for linear \emph{autonomous} difference-delay
equations, several representation
formulas stated in the literature (like \cite[Lemma 3.4]{Henry1974},
\cite[Chapter 12, Equation (3.16)]{Haleone} or \cite[Chapter 9, Theorem 1.2]{Hale} that deals with more general Volterra equations) seem to have issues\footnote{For instance, the formula in \cite[Lemma
  3.4]{Henry1974}, stated for time-invariant systems of the same type as
  \eqref{system_lin_formel} (possibly with infinitely many delays) does not agree with
  \eqref{formula_representation}. A check on systems with two delays will convince the
  reader that it is faulty because, probably due to misprinted indices,
  the integration is not carried out on the
  right interval. Still, our
    formula agrees with the one in \cite{Henry1974} in the case
    of a single delay ($N=1$)}, along with their proofs;
see Section~\ref{sec:VK} and Footnote~\ref{foot-le-pied}. For instance, when
  seeking to establish \eqref{formula_representation} for systems having periodic coefficients, which is a basic ingredient of the proofs in \cite{Soumis-SIMA}, the authors could not even come up with a satisfactory reference
 in the time-invariant case. The {\it raison d'être} for the present note is thus
  to present a result that  addresses the non-autonomous case, and
  at the same time that can be referenced even in the time-invariant setting. Our method of proof
  is in line   with the approach in \cite[Chapter 9, Sec. 1]{Hale}, trading \eqref{solution_fondamentale} for a Volterra equation and proving the existence of a resolvent for the latter. In this connection, Lemmata \ref{exresSV} and  \ref{solSVE} below
  are fundamental steps of the proof, that we could not locate in the literature and may be of independent interest.

\section{Summary on functions with bounded variations,  Lebesgue-Stieltjes integral and Volterra integral equations with $B^{\infty}$ kernel}
\label{sec:notation}

  In this section, we recall definitions and basic facts regarding functions of bounded variation
  and  Lebesgue-Stieltjes integrals, that the reader might want to consult
    before proceeding with Lemma \ref{exresSV},  Lemma  \ref{solSVE}
  and the proof of Theorem \ref{prop:fondamental}.

The real and complex fields are denoted by
  $\corpsR$ and $\corpsC$. For $d$ a strictly positive integer, we
  write $\|\cdot\|$ for the Euclidean norm on $\corpsC^d$
and  $\vertiii{\cdot}$ for the norm of a matrix $M \in \corpsC^{d
  \times d}$: $$\vertiii{M}=\sup_{\|x\|=1}\|Mx\|.$$

\subsection{Functions with bounded variations and the Lebesgue-Stieltjes integral}
\label{sub:sec1}
For $I$ a real interval and $f: I\to \mathbb{R}$ a function,
the \emph{total variation} of $f$ on $I$ is defined as
\begin{equation}
\label{defvar}
W_I(f):=\sup_{\stackrel{x_0<x_1<\cdots<x_N}{x_i\in I, N\in\mathbb{N}}}\sum_{i=1}^N|f(x_i)-f(x_{i-1})|.
\end{equation}
The space $BV(I)$ of \emph{functions with bounded variation} on $I$
consists of those $f$ such that  $W_I(f)<\infty$,
endowed with the  norm $\|f\|_{BV(I)}=W_{I}(f)+|f(y_0)|$ where $y_0\in I$ is arbitrary but fixed. 
Different $y_0$ give rise to equivalent norms for which
$BV(I)$ is a Banach space, and 
$\|.\|_{BV(I)}$ is stronger than the uniform norm.
We let $BV_r(I)$ and  $BV_l(I)$ be the closed subspaces
of $BV(I)$ comprised of right and left continuous functions, respectively.
We write $BV_{loc}(\mathbb{R})$ for the space of functions whose restriction to any bounded interval $I\subset \mathbb{R}$ lies in $BV(I)$.
Since
$
f(x_i)g(x_i)-f(x_{i-1})g(x_{i-1})$ $=\bigl(f(x_i)-f(x_{i-1})\bigr)g(x_i)+f(x_{i-1})\bigl(g(x_i)-g(x_{i-1})\bigr),
$
we observe from \eqref{defvar} that
\begin{equation}
\label{varprod}
W_I(fg)\leq W_I(f)\sup_{x\in I}|g(x)|+W_I(g)\sup_{x\in I}|f(x)|.
\end{equation}
Each $f\in BV(I)$ has a limit $f(x^-)$ (resp. $f(x^+)$) from the left (resp. right) at every $x\in I$
where the limit applies \cite[sec. 1.4]{lojasiewicz1988introduction}. Hence,
one can associate to $f$ a
finite signed Borel measure $\nu_f$ on $I$ such that
$\nu_f((a,b))=f(b^-)-f(a^+)$, and if $I$ is bounded on the right (resp. left) and contains its endpoint $b$ (resp. $a$), then $\nu_f(\{b\})=f(b)-f(b^-)$
(resp. $f(a^+)-f(a)$) \cite[ch. 7, pp. 185--189]{lojasiewicz1988introduction}.
Note  that different $f$ may generate the same $\nu_f$: for example 
 if $f$ and $f_1$ coincide except  at isolated interior points
 of $I$,  then $\nu_f=\nu_{f_1}$. For $g:I\to\mathbb{R}$ a measurable 
function  summable against $\nu_f$, 
the \emph{Lebesgue-Stieltjes integral} $\int gdf$ is defined as
$\int gd\nu_f$, whence the differential element $df$ identifies with $d\nu_f$
\cite[ch. 7, pp. 190--191]{lojasiewicz1988introduction}. This type of integral
is
useful  for it is suggestive of integration by parts, but  caution
must be used when integrating a function against $df$ over
a subinterval $J\subset I$ because the measure 
$\nu_{(f_{|J})}$ associated to the restricted function $f_{|J}$ needs \emph{not} coincide with the restriction $({\nu_f})_{|J}$ of the measure $\nu_{f}$ to $J$. More precisely, if the lower bound $a$ (resp. the upper bound $b$)
of $J$ belongs to $J$ and lies interior to $I$, then the two measures may differ by the weight they put on $\{a\}$ (resp. $\{b\}$), and they agree only when
$f$ is left (resp. right) continuous at $a$ (resp. $b$).
By $\int_J gdf$, we always mean that we integrate $g$ against $\nu_{(f_{|J})}$
and \emph{not} against  $({\nu_f})_{|J}$. As in the main formula \eqref{formula_representation},
we often trade the notation  $\int_J gdf$ for  one of the form
$\int_{a^\pm}^{b^\pm}gdf$, where  the interval of integration $J$ is encoded
in  the  bounds we put on the  integral sign: a lower bound $a^-$ (resp. $a^+$) means that
$J$ contains (resp. does not contain) its lower bound $a$, while an
upper bound $b^+$ (resp. $b^-$) means that
$J$ contains (resp. does not contain) its upper bound $b$.
Then, the previous word of caution applies to additive rules: for example,
when splitting $\int_{a^\pm}^{b^\pm}gdf$ into
  $\int_{a^\pm}^{c^\pm}gdf+\int_{c^\pm}^{b^\pm}gdf$ where $c\in(a,b)$, we must use $c^+$ (resp. $c^-$) if $f$ is right (resp. left) continuous at $c$.

  To a finite, signed or complex Borel measure $\mu$ on $I$,
  one can associate its \emph{total variation measure} $|\mu|$,
  defined on a Borel set $B\subset I$ by
$|\mu|(B)=\sup_{\mathcal{P}}\sum_{E\in\mathcal{P}}|\mu(E)|$ where $\mathcal{P}$ ranges over all partitions of $B$ into Borel sets, see \cite[sec. 6.1]{Rudin};
its total mass $|\mu|(I)$ is called the total variation of $\mu$,
denoted as $\|\mu\|$.
Thus, the total variation is defined both for functions of bounded variation
and for measures, with different meanings. When $f\in BV(I)$ is
monotonic then $W_{I}(f)=\|\nu_f\|$, but in general it only holds  that
 $\|\nu_f\|\leq2W_{I}(f)$; this follows from the Jordan decomposition of $f$
 as a difference of two increasing functions, each of which has variation at most
 $W_{I}(f)$ on $I$ \cite[Thm.\ 1.4.1]{lojasiewicz1988introduction}. 
 In any case, it holds that
  $|\int gdf|\leq \int|g|d|\nu_f|\leq 2W_I(f)\sup_I|g|$.
  The previous notations and definitions   also apply to 
vector and matrix-valued functions  $BV$-functions, replacing absolute values
in \eqref{defvar} by  Euclidean and matricial norms, respectively.

 \subsection{Volterra integral equations with kernels of type $B^{\infty}$}

\label{sec:VK}
%
\noindent\textit{Volterra equations for functions of a single variable
  have been studied extensively, see {\it e.g.}
  \textup{\cite{brunner2017volterra,gripenberg1990volterra}}. 
  However, the specific assumption that the kernel has bounded variation
  seems to have been treated somewhat tangentially. 
  On the one hand, it is subsumed in the measure-valued case presented
  in \textup{\cite[Ch.\ 10]{gripenberg1990volterra}}, but no convenient
  criterion is given there for the existence of a resolvent kernel.
  On the other hand, \textup{\cite[Ch.\ 9, Sec.\ 1]{Hale}} sketches the
  main arguments needed to handle kernels of bounded variation, but the
  exposition has issues\footnote{\label{foot-le-pied} For example, the
    equation for  $\widetilde{\rho}(t,s)$ stated at the top of page 258
    of that reference is
    not right.}.
}

We define a \emph{Stieltjes-Volterra  kernel of type $B^\infty$}  on $[a,b]\times[a,b]$ as  a  measurable 
function $\kappa:[a,b]\times[a,b]\to\mathbb{R}^{d\times d}$, with
$\kappa(t,\tau)=0$ for $\tau\geq t$,
such that the partial maps $\kappa(t,.)$ lie in 
$BV_l([a,b])$  and 
$\|\kappa(t,.)\|_{BV([a,b])}$ is uniformly
bounded with respect to $t\in[a,b]$. In addition, we require that
$\lim_{\tau\to t^-}W_{[\tau,t)}(\kappa(t,.))=0$ uniformly with respect to $t$; {\it i.e.}, to every $\varepsilon>0$, there exists $\eta>0$ such that
$W_{[\tau,t)}(\kappa(t,.))<\varepsilon$ as soon as $0<t-\tau<\eta$.
Note that $W_{[\tau,t)}(\kappa(t,.))\to0$ for fixed $t$ as $\tau\to t^-$ since $\kappa(t,.)$ has bounded variation on $[a,b]$, by the very definition \eqref{defvar}; so, the assumption really is  that the convergence is uniform
with respect to $t$. 
We endow the space $\mathcal{K}_{[a,b]}$ of such kernels with the norm 
$\|\kappa\|_{[a,b]}:=\sup_{t\in[a,b]}\|\kappa(t,.)\|_{BV([a,b])}$.
If $\kappa_k$ is a Cauchy sequence in $\mathcal{K}_{[a,b]}$,
then $\kappa_k$ converges uniformly on $[a,b]\times[a,b]$ to a $\mathbb{R}^{d\times d}$-valued function 
$\kappa$ because
\[
\vertiii{\kappa_k(t,\tau)-\kappa_l(t,\tau)}=\vertiii{(\kappa_k(t,\tau)-\kappa_l(t,\tau))-
(\kappa_k(t,t)-\kappa_l(t,t))}\leq\|\kappa_k(t,.)-\kappa_l(t,.)\|_{BV([a,b])}.
\]
Moreover, if $m$ is so large that
$\|\kappa_k-\kappa_l\|_{[a,b]}<\varepsilon$ for $k,l\geq m$ and
$\eta>0$ so small that $W_{[\tau,t)}(\kappa_m)<\varepsilon$ when $t-\tau<\eta$,
we get that $W_{[\tau,t)}(\kappa_l)\leq W_{[\tau,t)}(\kappa_m)+W_{[\tau,t)}(\kappa_m-\kappa_l)<2\varepsilon$, and letting $l\to\infty$ we get from
\cite[thm.\ 1.3.5]{lojasiewicz1988introduction} that $W_{[\tau,t)}(\kappa)\leq2\varepsilon$. The same reference implies that
$\|\kappa\|_{[a,b]}\leq \sup_k\|\kappa_k\|_{[a,b]}$, so that
$\kappa\in \mathcal{K}_{[a,b]}$. Finally, writing 
that $W_{[a,b]}(\kappa_k(t,.)-\kappa_l(t,.))<\varepsilon$ for each $t$ when
$k,l\geq m$
and passing to the limit as $l\to\infty$, we see that 
$\lim_k\|\kappa_k-\kappa\|_{[a,b]}=0$ whence
$\mathcal{K}_{[a,b]}$ is  a Banach space.
Note that a Stieltjes-Volterra kernel $\kappa$ of type $B^{\infty}$ is necessarily bounded with
$\sup_{[a,b]\times[a,b]}\vertiii{\kappa(t,\tau)}\leq\|\kappa\|_{[a,b]}$.

A \emph{resolvent} for the Stieltjes-Volterra kernel $\kappa$
on $[a,b]\times[a,b]$
is  a Stieltjes-Volterra  kernel $\rho$ on
$[a,b]\times[a,b]$ satisfying
\begin{equation}
\label{Stieltjesk}
\rho(t,\beta)=-\kappa(t,\beta)+\int_{\beta^-}^{t^-}d\kappa(t,\tau)\rho(\tau,\beta),
\qquad a\leq t,\beta \leq  b.
\end{equation}
We stress that $d\kappa(t,\tau)$ under the integral sign 
  is here the differential with respect to $\tau$ of a matrix-valued measure, and
  that $d\kappa(t,\tau)\rho(\tau,\beta)$   is a matrix
  product resulting in a matrix whose entries are sums of ordinary Lebesgue-Stieltjes integrands. Because matrices do not commute, a (generally) different result would be obtained upon writing $\int_{\beta^-}^{t^-}\rho(\tau,\beta) d\kappa(t,\tau)$. In all cases,
the repeated variable under the integral sign ($\tau$ in the present case) is a dummy one.

  The following two lemmata, proved in Section~\ref{proof_the_fond}, are the technical core of this paper.

%
%
\begin{lem}
\label{exresSV}
If $\kappa$ is a Stieltjes-Volterra kernel of type $B^{\infty}$ on $[a,b]\times[a,b]$,
a resolvent for $\kappa$ uniquely exists.
\end{lem}

\begin{lem}
\label{solSVE}
Let $\kappa$ be a Stieltjes-Volterra kernel of type $B^{\infty}$ on $[a,b]\times[a,b]$
and $\rho$ its resolvent.  For each $\mathbb{R}^d$-valued function
$g\in BV_r([a,b])$, the unique bounded measurable solution
to the equation 
\begin{equation}
\label{ap32}
y(t)=\int_{a^-}^{t^-} d\kappa(t,\tau) y(\tau)+g(t), \qquad a\leq t\leq b,
\end{equation}
is given by
\begin{equation}
\label{eqres2}
y(t)=g(t)-\int_{a^-}^{t^-}d\rho(t,\alpha) g(\alpha),\qquad a\leq t\leq  b.
\end{equation}
\end{lem}

The importance of Lemmata \ref{exresSV} and \ref{solSVE} stems from the fact that to prove our main result at the end of Section~\ref{proof_the_fond}, we will frame the difference delay
  equation \eqref{system_lin_formel} into a Stieltjes-Volterra integral equation with $B^{\infty}$ kernel.


\section{Proofs}

\label{proof_the_fond}

The proofs of Lemmata \ref{exresSV} and \ref{solSVE} are given before
the one of Theorem~\ref{prop:fondamental}, which relies on them.


\begin{proof}[Proof of Lemma~\ref{exresSV}]
Pick $r>0$ to be adjusted later, and for $\Psi\in\mathcal{K}_{[a,b]}$ let us 
define
\[F_r(\Psi)(t,\beta):=\int_{\beta^-}^{t^-}e^{-r(t-\tau)}d\kappa(t,\tau)\Psi(\tau,\beta),\qquad a\leq \beta,t\leq b.\] 
Then, $F_r(\Psi)(t,\beta)=0$ for $\beta\geq t$, and for $a\leq \beta_1<\beta_2< t$ 
we have that
\[
F_r(\Psi)(t,\beta_2)-F_r(\Psi)(t,\beta_1)=\!\int_{\beta_2^-}^{t^-}\!\!e^{-r(t-\tau)}d\kappa(t,\tau)\left(\Psi(\tau,\beta_2)-\Psi(\tau,\beta_1)\right)-\!\int_{\beta_1^-}^{\beta_2^-}\!\!e^{-r(t-\tau)}d\kappa(t,\tau)\Psi(\tau,\beta_1),
\]
where we used that $\kappa(t,.)$ is left continuous to assign the lower
(resp. upper)
bound $\beta_2^-$ to the first (resp. second) integral in the above right hand side. Now, the first integral  goes to $0$ as $\beta_1\to\beta_2$ by dominated convergence, as $\Psi(t,.)$ is left-continuous; the second integral also goes to
$0$, because 
$|\nu_{\kappa(t,.)}|([\beta_1,\beta_2))\to0$ when $\beta_1\to\beta_2$, since $\cap_{\beta_1\in[a,\beta_2)}[\beta_1,\beta_2)=\varnothing$. Altogether, we see that $F_r(\Psi)(t,.)$ is
left-continuous. Moreover, 
for $[c,d]\subset[a,t)$ and $c=\beta_0<\beta_1<\cdots<\beta_N=d$, one has:
\begin{align*}
  \sum_{i=1}^N\vertiii{F_r(\Psi)(t,\beta_i)-F_r(\Psi)(t,\beta_{i-1})}
  \hspace{-13em}&
  \\
                &\leq
\sum_{i=1}^N\vertiii{ \int_{\beta_i^-}^{t^-}
\!\!\!e^{-r(t-\tau)}\,d\kappa(t,\tau)  \left(\Psi(\tau,\beta_i)-\Psi(\tau,\beta_{i-1})\right)}
  +\sum_{i=1}^{N}\vertiii{ \int_{\beta_{i-1}^-}^{\beta_i^-}e^{-r(t-\tau)}d\kappa(t,\tau)\Psi(\tau,\beta_{i-1})}
  \\
  &\leq 
\sum_{i=1}^N\int_{\beta_i^-}^{t^-}
\!\!\!e^{-r(t-\tau)}\,d|\nu_{\kappa(t,.)}|\,\vertiii{(\Psi(\tau,\beta_i)-\Psi(\tau,\beta_{i-1}))}
    +\sum_{i=1}^N\int_{\beta_{i-1}^-}^{\beta_i^-}
\!\!\!e^{-r(t-\tau)}\,d|\nu_{\kappa(t,.)}|\,\vertiii{\Psi(\tau,\beta_{i-1})}
  \\&\leq 
\int_{d^-}^{t^-}
  \!\!d|\nu_{\kappa(t,.)}|\,\sum_{i=1}^N\vertiii{(\Psi(\tau,\beta_i)-\Psi(\tau,\beta_{i-1}))}
  \\&\qquad
+e^{-r(t-d)}\int_{c^-}^{d^-} 
  d|\nu_{\kappa(t,.)}|\sum_{i=1}^N\vertiii{(\Psi(\tau,\beta_i)-\Psi(\tau,\beta_{i-1}))}
  +\sup_{[a,t]\times[a,t]}\vertiii{\Psi}\int_{c^-}^{d^-}e^{-r(t-\tau)}d|\nu_{\kappa(t,.)}|
  \\&
\leq 
2\,W_{[d,t)}(\kappa(t,.))\sup_{\tau\in[d,t)}W_{[c,d]}(\Psi(\tau,.))
  \\&\qquad
  +
  2\,e^{-r(t-d)}W_{[c,d)}(\kappa(t,.))\sup_{\tau\in[c,d)}W_{[c,d]}(\Psi(\tau,.))
  +
2\,e^{-r(t-d)} \sup_{[a,t]\times[a,t]}\vertiii{\Psi}\,\,W_{[c,d)}(\kappa(t,.)).
\end{align*}
When $d=t$, the same inequality holds but then $W_{[d,t)}(\kappa(t,.))$ is zero.
Setting $c=a$ and $d=t$, we get from the above majorization that
$W_{[a,t]}(F_r(\Psi)(t,.))\leq 4\|\kappa\|_{[a,b]}\|\Psi\|_{[a,b]}$, and
since $F_r(\Psi)(t,\tau)=0$ for $\tau\geq t$ we deduce that
$W_{[a,b]}(F_r(\Psi)(t,.))=W_{[a,t]}(F_r(\Psi)(t,.))$ is bounded, uniformly with respect to $t$.
Next, if we fix $\varepsilon>0$ and 
pick $\eta>0$ so small that $W_{[\tau,t)}(\kappa(t,.))\leq\varepsilon$ as soon as $t-\tau\leq\eta$ (this is possible because $\kappa\in\mathcal{K}_{[a,b]}$),
the same  estimate yields 
\begin{equation}
\label{estvarFr}
W_{[c,t)}(F_r(\Psi)(t,.))
\leq 4W_{[c,t)}(\kappa(t,.)) \|\Psi\|_{[a,b]}\leq 4\,\varepsilon\,\|\Psi\|_{[a,b]},\qquad t-c\leq\eta.
\end{equation}
Altogether, we just showed that 
$F_r(\Psi)\in\mathcal{K}_{[a,b]}$. Moreover, if we take $r$ so large that 
$e^{-r\eta}<\varepsilon$, then either $t-a\leq \eta$ and then
\eqref{estvarFr} with $c=a$ gives us
$W_{[a,t)}(F_r(\Psi)(t,.)\leq 4\varepsilon\|\Psi\|_{[a,b]}$, or else
$t-\eta>a$ in which case \eqref{estvarFr}  with $c=t-\eta$, together with 
our initial estimate when
$c=a$ and $d=t-\eta$, team up to produce:
\begin{align}
  \nonumber
  W_{[a,t)}(F_r(\Psi))(t,.)&=W_{[a,t-\eta]}(F_r(\Psi)(t,.))+W_{[t-\eta,t)}(F_r(\Psi)(t,.))
  \\\nonumber
  &\leq
2\varepsilon\sup_{\tau\in[t-\eta,t)}W_{[a,t-\eta]}(\Psi(\tau,.))
+2\varepsilon
    W_{[a,t-\eta)}(\kappa(t,.))\sup_{\tau\in[a,t-\tau)}\!\!W_{[a,t-\eta]}(\Psi(\tau,.))
  \\\nonumber
  &\qquad+
2\varepsilon\sup_{[a,t]\times[a,t]}\vertiii{\Psi}W_{[a,t-\eta)}(\kappa(t,.))
    +4\varepsilon\|\Psi\|_{[a,b]}
  \\\label{estcontr}
&\leq 2\,\varepsilon\,\|\Psi\|_{[a,b]}\,\left(3+2\|\kappa\|_{[a,b]}\right).
\end{align}
Consequently, as $W_{[a,t)}(F_r(\Psi)(t,.))=W_{[a,t]}(F_r(\Psi)(t,.))$
by the left continuity of $F_r(\Psi)(t,.)$,
we can ensure upon choosing $r$ sufficiently large that
the operator $F_r:\mathcal{K}_{[a,b]}\to\mathcal{K}_{[a,b]}$ has arbitrary 
small norm. Hereafter, we fix $r$ so that $\vertiii{F_r}<\lambda<1$.

Now, let $\widetilde{\rho}_0=0$ and define inductively:
\[\widetilde{\rho}_{k+1}(t,\beta)=-e^{-rt}\kappa(t,\beta)+F_r(\widetilde{\rho}_{k})(t,\beta).
\]
Clearly $(t,\beta)\mapsto
e^{-rt}\kappa(t,\beta)$ lies in $\mathcal{K}_{[a,b]}$, so  that
$\widetilde{\rho}_k\in \mathcal{K}_{[a,b]}$ for all $k$. Moreover, we get from what precedes that 
$\|\widetilde{\rho}_{k+1}-\widetilde{\rho}_k\|_{[a,b]}\leq \lambda \|\widetilde{\rho}_{k}-\widetilde{\rho}_{k-1}\|_{[a,b]}$.
Thus, by the shrinking lemma, $\widetilde{\rho}_k$ converges in $\mathcal{K}_{[a,b]}$ 
to the unique $\widetilde{\rho}\in\mathcal{K}_{[a,b]}$ such that
\begin{equation}
\label{Stieltjeskm2}
\begin{split}
  \widetilde{\rho}(t,\beta)&=-e^{-rt}\kappa(t,\beta)+F_r(\widetilde{\rho})(t,\beta)
  \\[-.8ex]
  &=-e^{-rt}\kappa(t,\beta)+\int_{\beta^-}^{t^-}e^{-r(t-\tau)}d\kappa(t,\tau)\widetilde{\rho}(\tau,\beta),
  \quad a\leq t,\beta \leq b.
\end{split}
\end{equation}
Putting $\rho(t,\beta):=e^{rt} \widetilde{\rho}(t,\beta)$, one can see that
$\rho$ lies in $\mathcal{K}_{[a,b]}$ if and only if $\widetilde{\rho}$ does, 
and that \eqref{Stieltjeskm2} is equivalent to \eqref{Stieltjesk}. This achieves the proof.
\end{proof}


\begin{proof}[Proof Lemma~\ref{solSVE}]
By Lemma~\ref{exresSV} we know that $\kappa$ has a unique resolvent, say  $\rho$. Define $y$ through \eqref{eqres2} so that $y(a)=g(a)$, by inspection.
  Since $g\in BV_r([a,b])$ and  $\rho(t,\cdot)$, $k(t,.)$ lie
in $BV_l([a,b])$, an integration by parts
\cite[thm.\ 7.5.9]{lojasiewicz1988introduction}\footnote{Reference
  \cite{lojasiewicz1988introduction} restricts to integration by parts over open intervals only,
  and we do the same at the cost of a slightly lengthier computation.} using 
\eqref{Stieltjesk} along with Fubini's theorem
and the relations $\kappa(t,\alpha)=\rho(t,\alpha)=0$ for $\alpha\geq t$
gives us:
\[
\begin{array}{rl}
\int_{a^-}^{t^-}d\kappa(t,\alpha)y(\alpha)   \hspace{-5em}&  \hspace{4.5em}=
  (\kappa(t,a^+)-\kappa(t,a))y(a)+\int_{a^+}^{t^-}d\kappa(t,\alpha)y(\alpha)
  \\ &
  =(\kappa(t,a^+)-\kappa(t,a))g(a)+\int_{a^+}^{t^-}d\kappa(t,\alpha)g(\alpha)
  -
  \int_{a^+}^{t^-}d\kappa(t,\alpha)\int_{a^-}^{\alpha^-} d\rho(\alpha,\beta)g(\beta)
  \\ &
  =(\kappa(t,a^+)-\kappa(t,a))g(a)+\int_{a^+}^{t^-}d\kappa(t,\alpha)g(\alpha)-
  \int_{a^+}^{t^-}d\kappa(t,\alpha)\int_{a^+}^{\alpha^-}
  d\rho(\alpha,\beta)g(\beta)
  \\ &
 \ \ \ \ \ \ \ \  \ \ \ \ \ \ \ \ \ \ \ \ \ \ \ \ -\int_{a^+}^{t^-}d\kappa(t,\alpha)(\rho(\alpha,a^+)-\rho(\alpha,a))g(a)
  \\ &
=(\kappa(t,a^+) -\kappa(t,a))g(a)
  +\int_{a^+}^{t^-}d\kappa(t,\alpha)g(\alpha)
       +\int_{a^+}^{t^-}d\kappa(t,\alpha)\int_{a^+}^{\alpha^-}\rho(\alpha,\beta)dg(\beta)
  \\ &
   \ \ \ \ \ \ \ \  \ \ \ \ \ \ \ \ \ \ \ \ \ \ \ \ -\int_{a^+}^{t^-}d\kappa(t,\alpha)\left[\rho(\alpha,\beta)g(\beta)
  \right]_{\beta=a^+}^{\beta=\alpha^-}-\int_{a^+}^{t^-}d\kappa(t,\alpha)(\rho(\alpha,a^+)-\rho(\alpha,a))g(a)
  \\ &
  =(\kappa(t,a^+)-\kappa(t,a))g(a)+\int_{a^+}^{t^-}d\kappa(t,\alpha)g(\alpha)+\int_{a^+}^{t^-}\left(\int_{\beta^-}^{t^-}d\kappa(t,\alpha)\rho(\alpha,\beta)\right)dg(\beta)
  \\&
  \ \ \ \ \ \ \ \  \ \ \ \ \ \ \ \ \ \ \ \ \ \ \ \
  +
  \int_{a^+}^{t^-}d\kappa(t,\alpha)\rho(\alpha,a^+)g(a)
  -\int_{a^+}^{t^-}d\kappa(t,\alpha)(\rho(\alpha,a^+)-\rho(\alpha,a))g(a)\\ & 
 =(\kappa(t,a^+)-\kappa(t,a))g(a)+\int_{a^+}^{t^-}d\kappa(t,\alpha)g(\alpha) 
  \\ &
  \ \ \ \ \ \ \ \  \ \ \ \ \ \ \ \ \ \ \ \ \ \ \ \
  +\int_{a^+}^{t^-}
\bigl(\rho(t,\beta)+\kappa(t,\beta)\bigr)dg(\beta)
  +\int_{a^+}^{t^-}d\kappa(t,\alpha)\rho(\alpha,a)g(a)
 \\ &
=(\kappa(t,a^+)-\kappa(t,a))g(a)+\left[\kappa(t,\alpha)g(\alpha)\right]_{\alpha=a^+}^{\alpha=t^-}
+\int_{a^+}^{t^-}
\rho(t,\beta)dg(\beta)
+\int_{a^+}^{t^-}d\kappa(t,\alpha)\rho(\alpha,a)g(a)\\ &
  =-\kappa(t,a)g(a)
+\left[ \rho(t,\beta)g(\beta)  \right]_{\beta=a^+}^{\beta=t^-}
-\int_{a^+}^{t^-}
d\rho(t,\beta)g(\beta)+\int_{a^-}^{t^-}d\kappa(t,\alpha)\rho(\alpha,a)g(a)
  \\ &
=-\kappa(t,a)g(a)-\rho(t,a^+)g(a)
-\int_{a^+}^{t^-}
d\rho(t,\beta)g(\beta)
  +\bigl(\kappa(t,a)+\rho(t,a)\bigr)g(a)
  \\ &
=-\int_{a^-}^{t^-}d\rho(t,\beta)g(\beta)=y(t)-g(t).
\end{array}
\]
Thus, $y$ is a solution to \eqref{ap32}. Clearly, it is measurable, and
it is also bounded since $\|\rho(t,.)\|_{BV([a,b])}$ is  bounded independently of $t$ and $g$ is bounded. If  $\widetilde{y}$ is another solution to
\eqref{ap32} then $\widetilde{y}(a)=y(a)=g(a)$ by inspection, so that
$z:=y-\widetilde{y}$ is
a bounded measurable solution to the homogeneous equation:
\[
z(t)=\int_{a^+}^{t^-} d\kappa(t,\tau) z(\tau), \qquad a\leq t\leq b.
\]
Pick $r>0$ to be adjusted momentarily, and set $\widetilde{z}(t):=e^{-rt}z(t)$ so that
\begin{equation}
\label{SVhomt}
\widetilde{z}(t)=\int_{a^+}^{t^-}e^{-r(t-\tau)}d\kappa(t,\tau)\widetilde{z}(\tau).
\end{equation}
Let $\eta>0$ be so small that $W_{[\tau,t)}(\kappa(t,.))\leq1/4$ as soon as $t-\tau\leq\eta$; this is possible because $\kappa\in\mathcal{K}_{[a,b]}$. Then, it follows from \eqref{SVhomt} that for $t-\eta>a$:
\begin{align*}
  \left|\widetilde{z}(t)\right|
  &\leq\left|
  \int_{a^+}^{(t-\eta)^+}e^{-r(t-\tau)^+}d\kappa(t,\tau)\widetilde{z}(\tau)\right|
  +\left|
    \int_{(t-\eta)^+}^{t^-}e^{-r(t-\tau)^-}d\kappa(t,\tau)\widetilde{z}(\tau)\right|
  \\
  & \leq 2e^{-r\eta} W_{(a,t-\eta]}(\kappa(t,\cdot))\sup_{(a,t-\eta]}|\widetilde{z}|
  +\frac{1}{2}\sup_{(t-\eta,t)}|\widetilde{z}|\leq
  \sup_{(a,t)}|\widetilde{z}|\bigl(2e^{-r\eta}\|\kappa\|_{[a,b]}+\frac{1}{2}\bigr),
  \end{align*}
  while for $t-\eta\leq a$ we simply get
  $|\widetilde{z}(t)|\leq \sup_{(a,t)}|\widetilde{z}|/2$. Hence,
  choosing
  $r$ large enough, we may assume that $|\widetilde{z}(t)|\leq \lambda\sup_{(a,t)}|\widetilde{z}|$ for some $\lambda<1$ and all $t\in[a,b]$. Thus, if we choose $\lambda'\in(\lambda,1)$ and  $t_0\in(a,b]$, we can find $t_1\in(a,t_0)$ such that
  $|\widetilde{z}(t_1)|\geq(1/\lambda')|\widetilde{z}(t_0)|$,
  and proceeding inductively we construct a sequence $(t_n)$ in $(a,t_0]$ with
   $|\widetilde{z}(t_n)|\geq(1/\lambda')^n|\widetilde{z}(t_0)|$.
  If we had $|\widetilde{z}(t_0)|>0$, this would contradict the boundedness of
  $\widetilde{z}$, therefore $\widetilde{z}\equiv0$ on $(a,b]$, whence $z\equiv0$ so that $y=\widetilde{y}$.
\end{proof}


\begin{proof}[Proof of Theorem~\ref{prop:fondamental}]

It follows from \eqref{solution_fondamentale} that  
$\alpha\mapsto X(t,\alpha)$ lies in $BV_{loc}(\corpsR)$ for all $t$
and  satisfies
\begin{equation}  
 \label{solution_fondamentaled}
d_\alpha X(t,\alpha)=
\sum\limits_{j=1}^N D_j(t)\, d_\alpha X(t-\tau_j,\alpha) \qquad\mathrm{on}\quad
[s,s+\tau_j),\quad 1\leq j\leq N.
\end{equation}
Note,  since   
  $\alpha\mapsto X(t,\alpha)$ is left continuous (by \eqref{solution_fondamentale} again) while $[s,s+\tau_j)$ is open on the right,
  that $d_\alpha X(t,\alpha)$ and $d_\alpha X(t-\tau_j,\alpha)$
in \eqref{solution_fondamentaled} coincide with (the differential of)
the restrictions to
 $[s,s+\tau_j)$ of the (matrix-valued) measures
 $\nu_{X(t,\cdot)_{|[s,s+\tau_N)}}$ and
   $\nu_{X(t-\tau_j,\cdot)_{|[s,s+\tau_N)}}$,
          provided that $t\geq s+\tau_N$.

Now, substituting \eqref{solution_fondamentaled} in \eqref{formula_representation}
formally yields \eqref{system_lin_formel}
for $t\geq s+\tau_N$.
Hence, by  uniqueness of a 
solution $y$ to \eqref{system_lin_formel}  satisfying $y(s+\theta)=\phi(\theta)$ for $\theta \in [-\tau_N,0]$, it is enough to check 
\eqref{formula_representation} when $s\leq t< s+\tau_N$.
For this, we adopt the point of view
of reference \cite{Hale}, which is to construe delay systems as 
Stieltjes-Volterra equations upon representing delays by measures. More precisely, we can rewrite \eqref{system_lin_formel} as a Lebesgue-Stieltjes integral:
\begin{eqnarray}
\label{ap1}
y(t)=\int_{-\tau_N^-}^{0^-} d\mu(t,\theta) y(t+\theta),\qquad t\geq s,
\end{eqnarray}
with 
\begin{eqnarray}
\label{ap2}
\mu(t,\theta)= \sum\limits_{j=1}^ND_j(t)\mathfrak{H}(\theta+\tau_j),
\end{eqnarray}
where $y(\tau)$ is understood to be $\phi(\tau-s)$ when
$s-\tau_N\leq\tau\leq s$ and $\mathfrak{H}(\tau)$ is the Heaviside function
which is zero for $\tau\leq0$ and 1 for $\tau>0$, so that the associated measure on an interval of the form
$[0,a]$ or $[0,a)$ is a Dirac delta at $0$. Note that $\mathfrak{H}(0)=0$, which is not the usual
convention, but if we defined $\mathfrak{H}$ so that $\mathfrak{H}(0)=1$ then
expanding \eqref{ap1} using \eqref{ap2} would not give us back
\eqref{system_lin_formel} for
the term $D_N(t)y(t-\tau_N)$ would be missing. Observe also,
since $\tau_j>0$ for all $j$, that the minus sign in the upper bound of
the integral in \eqref{ap1} is immaterial and could be traded for a plus.
For $s\leq t\leq s+\tau_N$, singling out the initial data in \eqref{ap1}
yields
\begin{equation}
\label{redVE}
y(t)=\int_{(s-t)^-}^{0^-} d\mu(t,\theta) y(t+\theta)+f(t)
\quad\mathrm{with}\quad
f(t):=\int_{-\tau_N^-}^{(s-t)^-} d\mu(t,\theta) \phi(t+\theta-s), 
\end{equation}
where we took into account, when separating the integrals, that $\theta\mapsto\mu(t,\theta)$ is left continuous, while the integral over the empty interval is
understood to be zero.
It will be convenient to study \eqref{redVE} for $t\in[s,s+\tau_N]$, even though in the end the values of $y(t)$ only matter to us for $t\in[s,s+\tau_N)$. Define 
\begin{equation}
 \label{defSVk}
k(t,\tau):= \left\{
  \begin{array}{cl}
    \mu(t,\tau-t)-\sum_{j=1}^N D_j(t)&\mathrm{for}\ \tau\in [s,t],\\
 0 &\mathrm{for}\ \tau> t,   
  \end{array}\qquad t,\tau\in[s,s+\tau_N].
\right.\end{equation}
  
Note that $k(t,\tau)=0$ when  $t-\tau<\tau_1$, and
$d_\tau k(t,\tau)=d_\tau\mu(t,\tau-t)$ on $[s,s+\tau_N]$ for fixed $t$.  Hence,
\eqref{redVE} becomes
\begin{equation}
\label{ap3}
y(t)=\int_{s^-}^{t^-} dk(t,\tau) y(\tau)+f(t), \qquad s\leq t\leq s+\tau_N.
\end{equation}
Now, \eqref{ap3} is the Stieltjes-Volterra equation we shall work with.

It suffices to prove \eqref{formula_representation} 
for $t\in[s, s+\tau_N)$ under the additional assumption that
$\phi$ and the $D_j$, which are continuous by hypothesis, also have bounded 
and locally bounded variation, on $[-\tau_N,0]$ and $\mathbb{R}$ respectively. Indeed, functions of bounded variation
are dense in $C_s$ 
(say, because 
$C^1$-functions are), and if $\{\phi_k\}$ converges uniformly to $\phi$ in $C_s$ while $\{D_{j,k}\}$
converges uniformly to $D_j$ in $[s,s+\tau_N]$ as $k\to\infty$, then the 
solution to \eqref{system_lin_formel}  with initial condition
$\phi_k$ and coefficients $D_{j,k}$ converges uniformly on $[s,s+\tau_N]$ 
to the solution 
with initial condition $\phi$ and coefficients $D_j$, as is obvious by inspection. Hence, we shall assume without loss of generality that $\phi$ has bounded variation and the $D_j$ have locally bounded variation. 
Then, since \eqref{ap2} and \eqref{redVE} imply that
\begin{equation}
  \label{expf}
  f(t)=\sum_{\tau_\ell\in(t-s,\tau_N]} D_\ell(t)\phi(t-s-\tau_\ell), 
\end{equation}
it is clear from \eqref{expf} and \eqref{varprod} that
$f$ is in $BV_r([s,s+\tau_N])$.

As $\mathfrak{H}$ is left continuous and the $D_j$ are bounded
on $[s,s+\tau_N]$, it is easy to check that $k(t,\tau)$ defined in \eqref{defSVk} is a Stieltjes-Volterra  kernel of type $B^{\infty}$ on $[s,s+\tau_N]\times[s,s+\tau_N]$.
Let $\rho$ denote  the resolvent of this Stieltjes-Volterra kernel;  it exists thanks to Lemma~\ref{exresSV}.
As $f$ defined in \eqref{redVE}  lies in $BV_r([s,s+\tau_N])$,
the solution $y$ to \eqref{ap3} is given, in view of Lemma \ref{solSVE}, by
\begin{equation}
  \label{solvrs}
y(t)=f(t)-\int_{s^-}^{t^-}d\rho(t,\alpha) f(\alpha),\qquad s\leq t\leq  s+\tau_N.
\end{equation}
Since $\rho(t,\alpha)=0$ when $\alpha\geq t$,
the integral $\int_{s^-}^{t^-}$ can be replaced by $\int_{s^-}^{(s+\tau_N)^+}$ in
\eqref{solvrs}. Thus, putting $\delta_t$ for the Dirac delta distribution at $t$ and
$\widetilde{X}(t,\alpha):=\mathcal{H}(t-\alpha) I_d +\rho(t,\alpha)$ with
$\mathcal{H}(\tau)$ the ``standard'' Heaviside function which is
$0$ for $\tau<0$ and $1$ for $\tau\geq0$,
we deduce from \eqref{expf} and \eqref{solvrs}, since $d_\alpha \mathcal{H}(t-\alpha)=-\delta_t$
on $[s,s+\tau_N]$ for $s\leq t<s+\tau_N$,
that
\begin{eqnarray*}
  y(t)=\!-\!\int_{s^-}^{(s+\tau_N)^+}\!\!\!\!\!\!\!\!d\widetilde{X}(t,\alpha) f(\alpha)=\!-\!\int_{s^-}^{(s+\tau_N)^+}\!\!\!\!\!\!\!\!d\widetilde{X}(t,\alpha) \!\left(\sum_{\tau_\ell\in(\alpha-s,\tau_N]} \!\!D_\ell(\alpha)\phi(\alpha-s-\tau_\ell)\right)\!,\quad s\leq t< s+\tau_N.\\
  \end{eqnarray*}
  Rearranging,
  we get that
\begin{eqnarray*}
  y(t)=-\sum_{j=1}^N \int_{s^-}^{(s+\tau_j)^-}d\widetilde{X}(t,\alpha) D_j(\alpha)\phi(\alpha-s-\tau_j),\qquad s\leq t< s+\tau_N,\\
  \end{eqnarray*}
  which is what we want (namely:
  formula \eqref{formula_representation} for
  $s\leq t<s+\tau_N$) if only we can show that $\widetilde{X}(t,\alpha)$
  coincides with $X(t,\alpha)$ when  $\alpha\in [s,s+\tau_j)$ for each $j$ and every
  $t\in[s,s+\tau_N)$;
here, $X(t,\alpha)$ is defined by \eqref{solution_fondamentale}
where we set $s=\alpha$.

For this, we first observe that  $X(t,\alpha)=\widetilde{X}(t,\alpha)=0$ when
$\alpha>t$ and that  $X(t,t)=\widetilde{X}(t,t)=I_d$. Hence, we need only consider the case $\alpha\in[s,t)$ with $s<t<s+\tau_N$.
For $s\leq \alpha<t$, we get that
  \begin{eqnarray*}
    -k(t,\alpha)=k(t,t^-)-k(t,\alpha)
    =    \int_{\alpha^-}^{t^-}d k(t,\tau).
    \end{eqnarray*}
      Thus, \eqref{Stieltjesk} (where $\kappa=k$) in concert with the definition of
      $\widetilde{X}(t,\alpha)$
  imply that
  \begin{align*}
    \widetilde{X}(t,\alpha)
    &=I_d\,
                             \mathcal{H}(t-\alpha)-k(t,\alpha)+\int_{\alpha^-}^{t^-}dk(t,\tau)\rho(\tau,\alpha)
    \ =\ I_d +\int_{\alpha^-}^{t^-}dk(t,\tau)\bigl(I_d+\rho(\tau,\alpha)\bigr)
    \\
    &
    =I_d +\int_{\alpha^-}^{t^-}dk(t,\tau)\widetilde{X}(\tau,\alpha).
  \end{align*}
  Now, on $[\alpha,t)$, we compute from \eqref{ap2} and \eqref{defSVk}
   that $d_\tau k(t,\tau)=\sum_{t-\tau_j\geq\alpha}D_j(t)\delta_{t-\tau_j}$ and hence,
   since $\widetilde{X}(t-\tau_j,\alpha)=0$ when $\alpha>t-\tau_j$,
   the previous equation becomes:
   \begin{eqnarray}  
 \label{solution_fondamentale1t}
     \widetilde{X}(t,\alpha)=
I_d+\sum_{j=1}^N D_j(t) \widetilde{X}(t-\tau_j,\alpha) 
     \qquad \mbox{ for $s\leq \alpha<t$\ and\ $s\leq t <s+\tau_N$}.
   \end{eqnarray}
Comparing \eqref{solution_fondamentale1t} and \eqref{solution_fondamentale},
we see that  $\widetilde{X}(t,\alpha)$ and $X(t,\alpha)$ coincide
on $[s,s+\tau_N)\times[s,s+\tau_N)$, thereby ending the proof.
\end{proof}

  \section{Concluding remarks}
  We derived in this note a representation formula for linear non-autonomous
  difference-delay equations.

Note that a very different representation formula for the solutions of the same class of systems
appears in \cite[Section 3.2]{Chitour2016};
it gives an explicit combinatorial formula for the solutions in terms of sums of products of matrices
whose number of terms and factors increases with time.
In the present representation formula, this combinatorial complexity
is somehow encoded by an integral
formulation that stems out of interpreting System~\eqref{system_lin_formel} as a Volterra
integral equation, which provides a powerful tool to derive estimates.


As pointed out at the beginning of the paper, the present representation formula is 
  meant in the first place to help stability proofs; it
  plays a crucial role, for instance, in the derivation
of a necessary and sufficient condition for exponential stability
of periodic systems of the form \eqref{system_lin_formel},
see the manuscript \cite{Soumis-SIMA}.
Let us further stress  that the stability of more general classes of systems, like periodic differential delay systems ``of neutral
type''; {\it i.e.}, of the form
\begin{equation}
  \label{eq:1}
  \frac{\mathrm{d}}{\mathrm{d}t}\left(
y(t)-\sum_{j=1}^ND_j(t)y(t-\tau_j)
  \right)=B_0(t)y(t)+\sum_{j=1}^NB_j(t)y(t-\tau_j)
\end{equation}
(for some periodic matrices $B_0,\ldots,B_N$),
relies on the stability of \eqref{system_lin_formel} through perturbation arguments described
in \cite{Hale} in the time-invariant case.
It would be very interesting to extend the stability results from \cite{Soumis-SIMA} to such
differential systems of neutral type, and to this effect we feel
that a representation formula should be extremely useful.

\end{document}